\documentclass{amsart}
\usepackage{tikz-cd}
\usepackage{amsmath,amscd}
\usepackage{amsfonts}
\usepackage{mathtools}
\usepackage{amsmath,amssymb,amsthm,mathrsfs}
\usepackage[utf8]{inputenc}
\usepackage{hyperref}
\usepackage{enumitem}

\theoremstyle{plain} 
\newtheorem{theorem}{Theorem}
\numberwithin{theorem}{section}
\newtheorem*{theorem*}{Theorem}
\newtheorem{prop}[theorem]{Proposition}
\newtheorem{lemma}[theorem]{Lemma}
\newtheorem{coro}[theorem]{Corollary}

\newtheorem{definition}[theorem]{Definition}
\theoremstyle{definition} 

\newtheorem*{example}{Example}
\theoremstyle{remark} 
\newtheorem{remark}[theorem]{Remark}
\theoremstyle{definition}

\newcommand{\be}{\begin{equation} }
\newcommand{\ee}{\end{equation} }

\usepackage{color}

\newcommand{\bbC}{\mathbb{C}}
\newcommand{\C}{\mathbb{C}} 
\newcommand{\D}{\mathbb{D}}

\newcommand{\Z}{\mathbb{Z}}

\newcommand{\cA}{\mathcal{A}}
\newcommand{\cB}{\mathcal{B}}
\newcommand{\cC}{\mathcal{C}}

\newcommand{\cH}{\mathcal{H}}

\newcommand{\cM}{\mathcal{M}}
\newcommand{\cN}{\mathcal{N}}

\newcommand{\cQ}{\mathcal{Q}}

\newcommand{\sO}{\mathscr{O}}

\newcommand{\shD}{\mathscr{D}}

\newcommand{\Var}{\textup{Var}}
\newcommand{\Spec}{\textup{Spec}}

\newcommand{\DR}{\textup{DR}}

\newcommand{\bs}{{\bf s}}

\DeclareMathOperator{\can}{can}

\DeclareMathOperator{\Cone}{Cone}
\newcommand{\et}{\end{thm}}
\newcommand{\bde}{\begin{de}}

\newcommand{\ede}{\end{de}}
\newcommand{\bc}{\begin{cor}}
\newcommand{\ec}{\end{cor}}
\newcommand{\blm}{\begin{lem}}
\newcommand{\elm}{\end{lem}}
\newcommand{\bp}{\begin{proof}}
\newcommand{\ep}{\end{proof}}
\newcommand{\beq}{\begin{equation*}\label{xx}}
\newcommand{\eeq}{\end{equation*}}

\renewcommand{\nu}{{\mathcal{V}}}

\renewcommand{\ref}[1]{\hyperref[#1]{\ref*{#1}}}

\newcommand{\nearbyd}[2]{\Psi_{f, #2}{#1}}
\newcommand{\nearbydp}[2]{\psi_{f, #2}{#1}}
\newcommand{\vanish}[1]{\Phi_f{#1}}
\newcommand{\vanishp}[1]{\phi_f{#1}}

\title{On the comparison of nearby cycles via $b$-functions}
\author{Lei Wu}
\address{Department of Mathematics, University of Utah.
155 S 1400 E,  Salt Lake City, UT 84112, USA}
\email{lwu@math.utah.edu}


\begin{document}

\begin{abstract}
    In this article, we give a simple proof of the comparison of nearby and vanishing cycles in the sense of Riemann-Hilbert correspondence following the idea of Beilinson and Bernstein, without using the Kashiwara-Malgrange $V$-filtrations. 
\end{abstract}

\maketitle

\section{Introduction}
The idea of the nearby and vanishing cycles can be traced back to Grothendieck and they are first introduced by Deligne \cite{De}. Nearby and vanishing cycles are widely studied from different perspectives, for instance by Beilinson \cite{BeiGP} algebraically and Kashiwara and Schapira \cite{KS} under the microlocal setting. They are also very useful, for instance Saito \cite{Sai88} used nearby and vanishing cycles to give an inductive definition of pure Hodge modules.  

Using the so-called Kashiwara-Malgrange filtration, as well as its refinement, the $V$-filtration, Kashiwara \cite{KasV} defined nearby and vanishing cycles for holonomic $\shD$-modules and proved a comparison theorem in the sense of Riemann-Hilbert correspondence (see \cite[Theorem 2]{KasV}). 

Beilinson and Bernstein constructed the unipotent (or more precisely, nilpotent) nearby and vanishing cycles for holonomic $\shD$-modules using $b$-functions under the algebraic setting in \cite[\S 4.2]{BB93} by using the complete ring $\bbC[[t]]$; see also \cite[\S 2.4]{BG} by using localization of $\bbC[t]$ without completion. One then can ``glue" the open part and the vanishing cycle of a holonomic $\shD$-module along any regular functions in the sense of Beilinson \cite{BeiGP} (see also \cite[Theorem 4.6.28.1]{Gilnotes} and \cite{Lich}).

In this article, by using the theory of relative holonomic $\shD$-modules by Maisonobe \cite{Mai}
and its development in \cite{WZ,BVWZ1,BVWZ2}, we give a slight refinement of the construction of nearby and vanishing cycles of Beilinson and Bernstein to other eigenvalues. Then we will give a simple proof of the comparison of nearby and vanishing cycles in the sense of Riemann-Hilbert correspondence without using $V$-filtrations (see Theorem \ref{thm:maincom}), which we believe is new. However, essentially the proof has been hinted by Beilinson and Bernstein (see for instance \cite[Remark 4.2.2(v)]{BB93}). Another point of the proof of the comparison is that since the proof is purely algebraic, it can be transplanted on smooth varieties over fields of characteristic $0$ without much modification.

Let $X$ be a smooth algebraic variety over $\bbC$ (or a complex manifold) and $f$ a regular function on $X$ (or a holomorphic function on $X$) and let $\cM$ be a holonomic $\shD_X$-module.  For $\alpha\in\bbC$, we denote the $\alpha$-nearby cycle of $\cM$ along $f$ by $\nearbyd{\cM}{\alpha}$ and denote the vanishing cycle of $\cM$ along $f$ by $\vanish{\cM}$ (see \S \ref{subsec:nbvan} for definitions). The sheaves $\nearbyd{\cM}{0}$ and $\vanish{\cM}$ are the same as $\Psi^{\textup{nil}}(\cM)$ and $\Phi^{\textup{nil}}(
\cM)$ in \cite[\S 2.4]{BG}.

From construction, $\nearbyd{\cM}{\alpha}$ and $\vanish{\cM}$ have the action by $s$, where $s$ is the independent variable introduced in defining $b$-functions (see \S \ref{subsec:bf}) and $\nearbyd{\cM}{\alpha}$ only depends on 
$\cM|_U$.

The $b$-function is also called the Bernstein-Sato polynomial. At least for $\cM=\sO_X$, there are algorithms to compute $b$-functions with the help of computer program (for instance Singular and Macaulay2). On the contrary, Kashiwara-Malgrange filtrations are more difficult to deal with from algorithmic perspectives as far as we know. Therefore, it seems easier to deal with nearby and vanishing cycles via $b$-functions.   

Following Beilinson's idea in \cite{BeiGP}, we define the nearby and vanishing cycles for $\bbC$-perverse sheaves by using Jordan blocks. Let $K$ be a perverse sheaf of $\bbC$-coefficients on $X$. One can also work with perverse sheaves with arbitrary fields of coefficients, see for instance \cite{Rei}. When we talk about perverse sheaves on an algebraic variety over $\bbC$, we use the Euclidean topology by default. If one wants to work with algebraic varieties with other base fields of characteristic zero, then one can consider \'etale sheaves. 

For $\lambda\in \bbC^*$ we define the $\lambda$-nearby cycle by 
\[\nearbydp{K}{\lambda}\coloneqq \lim_{m\to\infty}i^{-1}Rj_*(K|_U\otimes f^{-1}_0L^{1/\lambda}_m),\]
where $j: U=X\setminus D\hookrightarrow X$ and $D$ is the divisor defined by $f=0$, $i\colon D\hookrightarrow X$ and $f_0=f|_U$. Here $L^\lambda_m$ is isomorphic to the local system given by a $m\times m$ Jordan block with the eigenvalue $\lambda^{-1}$ on $\bbC^*$ (see \S \ref{subsec:exjb} for the construction). For all $m\in \Z$, $L^\lambda_m$ naturally form a direct system with respect to the natural order on $\Z$. The vanishing cycle of $K^\bullet$ along $f$ is then defined by 
\[\vanishp{K}\coloneqq \Cone(i^{-1} K\rightarrow \nearbydp{K}{1}).\]
We then have a canonical morphism
\[\can\colon \nearbydp{K}{1}\to  \vanishp{K}\]
fitting in the tautological triangle
\[i^{-1}K\rightarrow \nearbydp{K}{1}\xrightarrow{\can}  \vanishp{K}\xrightarrow{+1}.\]
The monodromy action on $L^{1/\lambda}_m$ naturally induces the monodromy action on both $\nearbydp{K}{\lambda}$ and $\vanishp{K}$, denoted by $T$. By construction, $T-\lambda$ acts on $\nearbydp{K}{\lambda}$ nilpotently.
When $\lambda=1$, $\log T_u$ induces 
\[\Var\colon  \vanishp{K}\rightarrow\nearbydp{K}{1}.\]
See \S \ref{sec:nvp} for details. The above definition of nearby and vanishing cycles coincides with Deligne's construction (see \cite[Chapter VI. 6.4.6]{Bj} for $\nearbydp{K}{1}$ and \cite[\S 3]{Wnv} and \cite{Rei} in general). The morphism $\Var$ corresponds to the ``Var" morphism defined in \cite{KasV}. 

The rest of this paper is mainly about the proof of the following comparison theorem. 
\begin{theorem}\label{thm:maincom}
Assume that $\cM$ is a regular holonomic $\shD_X$-module. Then we have 
\[\DR(\nearbyd{\cM}{\alpha})\simeq i_*\nearbydp{\DR(\cM)}{\lambda}[-1]\]
for every $\alpha\in \bbC$ where $\lambda=e^{2\pi \sqrt{-1}\alpha}$ and the monodromy action $T$ corresponds to 
\[T\simeq \DR(e^{-2\pi \sqrt{-1}s})=\DR(\lambda\cdot e^{-2\pi \sqrt{-1}(s+\alpha)})\]
under the isomorphisms $($since $s+\alpha$ acts on $\nearbyd{\cM}{\alpha}$ nilpotently$)$, and 
\[\DR(\vanish{\cM})\simeq i_*\vanishp{\DR(\cM)}[-1]\]
where $\DR$ denotes the de Rham functor for $\shD$-modules.
Furthermore, we have natural morphisms of $\shD$-modules
\[v: \vanish{\cM}\to \nearbyd{\cM}{0} \textup{ and } c:\nearbyd{\cM}{0}\to \vanish{\cM}\]
so that there exists an isomorphism of quivers
\[\begin{tikzcd}[every arrow/.append style={shift left}]
 \DR\big( \nearbyd{\cM}{0}  \arrow{r}{c} & \vanish{\cM}\arrow{l}{v}\big)\simeq i_*\big(\psi_{f, 1}\big(\DR(\cM)\big) \arrow{r}{\can}  &\phi_{f}\big(\DR(\cM)\big)[-1] \arrow{l}{\Var}\big).
 \end{tikzcd}\]
\end{theorem}
\begin{remark}
In Theorem \ref{thm:maincom}, we see that $\nearbyd{\cM}{\alpha+k}$ correspond to a unique nearby cycle of $\DR(\cM)$ for all $k\in \Z$. Namely, the Riemann-Hilbert correspondence for nearby cycles is $\Z$-to-1. However, $\nearbyd{\cM}{\alpha+k}$ is unique up to the $t$-action by Eq.\eqref{eq:tactnb}.
\end{remark}

\subsection*{Acknowledgement}
The author thanks P. Zhou for useful comments.

\section{Nearby and vanishing cycles for holonomic $\shD$-modules}\label{sec:nvd}
\subsection{$b$-function and Localization}\label{subsec:bf}
We recall the construction of $b$-functions. Let $X$ be a smooth algebraic variety over $\bbC$ of dimension $n$ and let $f$ be a regular function on $X$. We denote by $D$ the divisor defined by $f=0$, by $j: U= X\setminus D\hookrightarrow X$ the open embedding and by $i\colon D\hookrightarrow X$ the closed embedding. We assume that $\cM_U$ is a (left) holonomic $\shD_U$-module so that 
\[\cM_U=\shD_U\cdot \cM_0|_U\]
for some fixed coherent $\sO_X$-submodule $\cM_0\subseteq j_*(\cM_U)$ throughout this section. 
We then introduce an independent variable $s$ and consider the free $\bbC[s]$-module
\[j_*(\cM_U[s]\cdot f^s)=j_*(\cM_U)\cdot f^s\otimes_\bbC\bbC[s].\]
The module $j_*(\cM_U[s]\cdot f^s)$ has a natural $\shD_X[s]$-module structure by requiring 
\[v(f^s)=sv(f)f^{s-1},\]
for any vector field $v$ on $X$, where $\shD_X[s]\coloneqq \shD_X\otimes_\bbC\bbC[s]$. Notice that the module $j_*(\cM_U[s]\cdot f^s)$ is not necessarily coherent over $\shD_X[s]$. We then consider the coherent $\shD_X[s]$-submodule generated by $\cM_0\cdot f^{s+k}$
\[\shD_X[s]\cM_0\cdot f^{s+k}\subseteq j_*(\cM_U[s]\cdot f^s)\]
for every $k\in\Z$. It is obvious that we have inclusions
\[\shD_X[s]\cM_0\cdot f^{s+k_1}\subseteq \shD_X[\bs]\cM_0\cdot f^{s+k_2}\]
when $k_1\ge k_2$.

\begin{definition}[$b$-function]
The $b$-function of $\cM_U$ along $f$, also called the Bernstein-Sato polynomial, is the monic polynomial $b(s)\in \bbC[s]$ of the least degree so that 
\[b(s) \textup{ annihilates } \dfrac{\shD_X[s]\cM_0\cdot f^s}{\shD_X[s]\cM_0\cdot f^{s+1}}.\]
\end{definition}
In particular, if we pick $\cM_U=\sO_U$ and $\cM_0=\sO_X$, then the above definition gives us the usual $b$-function for $f$ (see for instance \cite{Kasbf}). From definition, the roots of the $b$-function of $\cM_U$ depends on the choice of $\cM_0$. However, we will see that an arithmetic set generated by the roots is independent with the choice. 
\begin{remark}
In the case that $X$ is a complex manifold and $f$ is a holomorphic function on $X$, for an analytic holonomic $\shD_X$-module $\cM$, one can use $\cM(*D)$, the algebraic localization of $\cM$ along $D$, to replace $j_*(\cM_U)$ and define $b$-functions in the analytic setting in a similar way. 
\end{remark}
\begin{theorem}[Bernstein and Sato]
The $b$-functions along $f$ exist for holonomic $\shD_U$-modules.
\end{theorem}
The above theorem for $\sO_U$ is due to Bernstein algebraically and Sato analytically. Bj\" ork extended it for arbitrary holonomic modules in the analytic setting (see \cite[Chapter VI]{Bj}).
\begin{definition}[Localization]
Assume that $\cN$ is a $($left$)$ coherent $\shD_X[s]$-module and $q$ is a prime ideal in $\bbC[s]$. Then we define the localization of $\cN$ at $q$ by 
\[\cN_q=\cN\otimes_{\bbC[s]}\bbC[s]_q\]
where $\bbC[s]_q$ is the localization of $\bbC[s]$ at $q$. In particular, if $q$ is the ideal generated by $0\in \bbC[s]$ $($i.e. $q$ is the generic point of $\bbC=\Spec\bbC[s]$$)$, then $\cN_q$ becomes a coherent $\shD_{X(s)}$-module, where $X(s)$ is the variety defined over $\bbC(s)$ of $X$ after the base change $\bbC\to \bbC(s)$, where $\bbC(s)$ is the fractional field of $\bbC[s]$. 

We write the localization of $\shD_X[s]\cM_0\cdot f^{s+k}$ and $j_*(\cM_U[s]\cdot f^s)$ at $m$ by 
\[\shD_X[s]_m\cM_0\cdot f^{s+k}\textup{ and }j_*(\cM_U[s]_m\cdot f^s)\]
respectively for a maximal ideal $m\subseteq \bbC[\bs]$ and by
\[\shD_X(s)\cM_0\cdot f^{s+k}\textup{ and }j_*(\cM_U(s)\cdot f^s)\]
the localization at the generic point.
\end{definition}
\begin{definition}[Duality]
Assume that $\cN$ is a $($left$)$ coherent $\shD_X[s]$-module and $\cN_q$ is a coherent $\shD_X[s]_q$-module for a prime ideal $q\subseteq\bbC[s]$. We then define the duality by 
\[\D(\cN)\coloneqq\mathcal Rhom_{\shD_X[s]}(\cN,\shD_X[s])\otimes_\sO \omega^{-1}_X[n],\]
and 
\[\D(\cN_q)\coloneqq\mathcal Rhom_{\shD_X[s]_q}(\cN_q,\shD_X[s]_q)\otimes_\sO \omega^{-1}_X[n]\]
where $\omega_X$ is the dualizing sheaf of $X$. The twist by $\omega_X$ is to make the dual of $\cN$ $($resp. $\cN_q$$)$ a complex of left $\shD_X[s]$-modules $($resp. $\shD_X[s]_q$-modules$)$.

In the case that $\D(\cN)$ $($resp. $\D(\cN_q)$$)$ has only the zero-th cohomological sheaf non-zero, we also use $\D(\cN)$ $($resp. $\D(\cN_q)$$)$ to denote $\cH^0(\D(\cN))$ $($resp. $\cH^0(\D(\cN_q))$$)$.
\end{definition}
Since the variable $s$ is in the center of $\shD_X[s]$, one can easily check that duality and localization commute, i.e.
\be\label{eq:commDL}
\D(\cN)_q\simeq \D(\cN_q).
\ee

We can evaluate $\cN$ at the residue field of a maximal ideal $m\subseteq \bbC[s]$:
\[\cN\otimes^L_{\bbC[s]}\bbC_m\]
where $\bbC_m\simeq\bbC$ is the residue field $\bbC[s]/m$ and the $\otimes^L_{\bbC[s]}$ denotes the derived tensor functor over $\bbC[s]$; it gives a complex of coherent $\shD_X$-modules. Furthermore, since $\shD_X[s]$ is free over $\bbC[s]$, one can check that evaluation and duality commute, i.e.
\be\label{eq:commED}
\D(\cN)\otimes^L_{\bbC[s]}\bbC_m \simeq \D(\cN\otimes^L_{\bbC[s]}\bbC_m),
\ee
where the second $\D$ denotes the duality functor for complexes of coherent $\shD$-modules. Because of the evaluation functor and its commutativity with duality, we also call $\shD_X[s]$-modules the relative $\shD$-modules over $\bbC[s]$. See \cite[\S 5]{WZ} for further discussions of relative $\shD$-modules for the multi-variate $s$ and also \cite[\S 3]{BVWZ1} in general.

The following lemma is obvious to check; see also \cite[Lemma 5.3.1]{BVWZ2} for a multi-variate version.
\begin{lemma}\label{lm:opendual}
We have 
\[\D(\cM_U[s]\cdot f^s)\simeq \D(\cM_U)[s]\cdot f^{-s}\simeq \D(\cM_U)[s]\cdot f^{s}\]
where the last isomorphism is given by substituting $s$ by $-s$ (and hence it is not canonical).
\end{lemma}

\begin{lemma}\label{lm:cm}
The $\shD_X[s]$-module $\shD_X[s]\cM_0\cdot f^{s+k}$ is $n$-Cohen-Macaulay for every $k\in \Z$, i.e. the complex $\D(\shD_X[s]\cM_0\cdot f^{s+k})$ only has one non-zero cohomology sheaf 
\[\cH^0(\D(\shD_X[s]\cM_0\cdot f^{s+k}))\simeq \mathcal Ext^n(\shD_X[s]\cM_0\cdot f^{s+k},\shD_X[s])\otimes_\sO\omega_X^{-1}.\]
\end{lemma}
\begin{proof}
By \cite[Proposition 14]{Mai} (taking $p=1$), we see that $\shD_X[s]\cM_0\cdot f^{s+k}$ is $n$-pure (see for instance \cite[\S 4]{BVWZ1} for the definition of purity). Moreover, by \cite[R\'esultat 2]{Mai}, $\shD_X[s]\cM_0\cdot f^{s+k}$ is relative holonomic (see \cite[Definition 3.2.3]{BVWZ1}). By \cite[Theorem 3.2.2]{BVWZ1}, since we have a single $s$, $n$-purity is equivalent to $n$-Cohen-Macaulayness for relative holonomic modules over $\bbC[s]$. The proof is then done.  
\end{proof}
By the above lemma and the isomorphism \eqref{eq:commDL}, we immediately have:
\begin{coro}\label{cor:localcm}
For every prime ideal $q\subseteq \bbC[s]$, the $\shD_X[s]_q$-module $\shD_X[\bs]_q\cM_0\cdot f^{s+k}$ is $n$-Cohen-Macaulay for every $k\in \Z$, i.e. the complex $\D(\shD_X[s]_q\cM_0\cdot f^{s+k})$ only has one non-zero cohomology sheaf 
\[\cH^0(\D(\shD_X[s]_q\cM_0\cdot f^{s+k}))\simeq \mathcal Ext^n(\shD_X[s]_q\cM_0\cdot f^{s+k},\shD_X[s]_q)\otimes_\sO\omega_X^{-1}.\]
\end{coro}
The above corollary is the same as \cite[Lemma 2(a) and Corollary 3]{BG}. But our proof (by using Lemma \ref{lm:cm}) is different from the approach in \emph{loc. cit.} See also \cite[\S 5]{WZ} for the multi-variate generalization. 

For every $\alpha\in \bbC$, we denote by $m_\alpha$ the maximal ideal of $\alpha$ in $\bbC[s]$, that is, the ideal generated by $s-\alpha$, and $\bbC_\alpha$ its residue field.
\begin{lemma}\label{lm:bj_*}
We have 
\[\shD_X(s)\cM_0\cdot f^{s-k}=j_*(\cM_U(s)\cdot f^s)\]
for every $k\in \Z$. Moreover, for every $\alpha\in \bbC$, there exists $k_0>0$ so that
\[\shD_X[s]_{m_\alpha}\cM_0\cdot f^{s-k}=j_*(\cM_U[s]_{m_\alpha}\cdot f^s)\]
for all $k>k_0$. 
\begin{proof}
We write by $b(s)$ the $b$-function of $\cM_U$. Since $b(s+k)$ is invertible in $\bbC(s)$, we have $\shD_X(s)\cM_0\cdot f^{s+k}=\shD_X(s)\cM_0\cdot f^{s}$ for every $k\in\Z$. Hence, the first statement follows. The second statement can be proved similarly. 
\end{proof}
\end{lemma}
We define $j_!$-extensions
\[j_!(\cM_U(s)\cdot f^s)\coloneqq \D\circ j_*\circ \D(\cM(s)\cdot f^s)\]
and 
\[j_!(\cM_U[s]_{m_\alpha}\cdot f^s)\coloneqq \D\circ j_*\circ \D(\cM[s]_{m_\alpha}\cdot f^s)\]
for every $\alpha\in \bbC$. By Lemma \ref{lm:opendual}, Corollary \ref{cor:localcm} and Lemma \ref{lm:bj_*} (both for $\D(\cM_U)$), they are both sheaves (instead of complexes).

Since $\D\circ\D$ is identity, using the adjunction pair $(j^{-1}, j_*)$, we have natural morphisms  
\[j_!(\cM_U(s)\cdot f^s)\rightarrow j_*(\cM_U(s)\cdot f^s)\]
and 
\[j_!(\cM_U[s]_{m_\alpha}\cdot f^s)\rightarrow j_*(\cM_U[s]_{m_\alpha}\cdot f^s)\]
for every $\alpha\in \bbC$.

For every $\alpha\in \bbC$, the multi-valued function $f^\alpha$ gives a local system on $U$. We then denote by $\cM_U\cdot f^\alpha$ the holonomic $\shD_U$-module twisted by the local system given by $f^\alpha$. It is obvious by construction that $\cM_U\cdot f^\alpha$ is $\Z$-periodic, that is, 
\be\label{eq:twistperiodic}
\cM_U\cdot f^\alpha=\cM_U\cdot f^{\alpha+k}
\ee
for every $k\in \Z$.
\begin{example}
For some $\alpha\in \bbC$, consider the regular holonomic module 
\[\bbC[t,1/t]\cdot t^\alpha\]
by assigning $t\partial_t\cdot t^\alpha=\alpha t^\alpha$, where $t$ is the complex coordinate of the complex plane $\bbC$ and $t^\alpha$ is the symbol of the multivalued function ``$t^\alpha$". Then the multi-valued flat section on $\bbC^*$, the punctured complex plane, is $e^{-\alpha\log t}\cdot t^\alpha$. Consequently, the monodromy $T$ of the underlying rank 1 local system (around the origin counterclockwise) is the multiplication by $e^{-2\pi\sqrt{-1}\alpha}$, by choosing different branches of $\log t$.
\end{example}

By using the Deligne-Goresky-MacPherson extension (or the minimal extension), the following theorem is first proved by Ginsburg in \cite[\S 3.6 and 3.8]{Gil}, as well as in \cite{BG}, which is essentially due to Beilinson and Bernstein. See also \cite[Theorem 5.3]{WZ} for the multi-variate generalization.
\begin{theorem}[Beilinson and Bernstein]\label{thm:BB}
We have: 
\begin{enumerate}
    \item the natural morphism $j_!(\cM_U(s)\cdot f^s)\rightarrow j_*(\cM_U(s)\cdot f^s)$ is isomorphic and they are both equal to $\shD_X(s)\cM_0\cdot f^{s+k}$ for every $k\in \Z$;
    \item the natural morphism $j_!(\cM_U[s]_{m_\alpha}\cdot f^s)\rightarrow j_*(\cM_U[s]_{m_\alpha}\cdot f^s)$ is injective for every $\alpha\in \bbC$;
    \item for every $\alpha\in \bbC$, there exists $k_0\in \Z_{+}$ so that for all $k>k_0$ we have 
    \[j_*(\cM_U[s]_{m_\alpha}\cdot f^s)=\shD_X[s]_{m_\alpha}\cM_0\cdot f^{s-k},\]
    \[j_!(\cM_U[s]_{m_\alpha}\cdot f^s)=\shD_X[s]_{m_\alpha}\cM_0\cdot f^{s+k},\]
    \[j_*(\cM_U\cdot f^\alpha)\stackrel{q.i.}{\simeq}\shD_X[s]_{m_\alpha}\cM_0\cdot f^{s-k}\otimes^L_{\bbC[\bs]} \bbC_\alpha \]
    and 
    \[j_!(\cM_U\cdot f^\alpha)\stackrel{q.i.}{\simeq}\shD_X[s]_{m_\alpha}\cM_0\cdot f^{s+k}\otimes^L_{\bbC[\bs]} \bbC_\alpha;\]
    \item for every $\alpha\in \bbC$, if $\alpha+k$ is not a root of the $b$-function of $\cM_U$ for every $k\in \Z$, then we have 
    \[j_!(\cM_U[s]_{m_\alpha}\cdot f^s)= j_*(\cM_U[s]_{m_\alpha}\cdot f^s)=\shD_X[s]_{m_\alpha}\cM_0\cdot f^{s+k}\]
    and 
    \[j_!(\cM_U\cdot f^\alpha)= j_*(\cM_U\cdot f^\alpha)\stackrel{q.i.}{\simeq} \shD_X[s]_{m_\alpha}\cM_0\cdot f^{s+k}\otimes^L_{\bbC[s]} \bbC_\alpha \]
    for every $k\in \Z$, where $q.i.$ stands for quasi-isomorphism.
\end{enumerate}
\end{theorem}
\subsection{Nearby and vanishing cycles}\label{subsec:nbvan}
We now give constructions of nearby and vanishing cycles. We continue using the notations and setups in \S\ref{subsec:bf}. We assume that $\cM$ is a holonomic $\shD_X$-module so that $\cM|_U\simeq\cM_U$.
\begin{definition}
For every $\alpha\in \bbC$, the $\alpha$-nearby cycle of $\cM$ is 
\[\nearbyd{\cM}{\alpha}\simeq\nearbyd{\cM_U}{\alpha}=\dfrac{j_*(\cM_U[s]_{m_{-\alpha}}\cdot f^s)}{j_!(\cM_U[s]_{m_{-\alpha}}\cdot f^s)}.\]
\end{definition}
The above definition needs Theorem \ref{thm:BB} (2) to get the quotient. From definition, the $\alpha$-nearby cycle of $\cM$ only depends on $\cM|_U$. 

Recall that $\shD_X[s]\cM_0\cdot f^s$ has a $t$-action given by 
\[t\cdot s=s+1.\]
By definition, $t$ acts on $\nearbyd{\cM}{\alpha}$ and 
\be\label{eq:tactnb}
t\cdot \nearbyd{\cM}{\alpha}=\nearbyd{\cM}{\alpha+1}.
\ee

We define $\Lambda$ by the discrete set: $\Z-$roots of the $b$-function of $\cM_U$. By Eq.\eqref{eq:twistperiodic}, $\Lambda$ is independent of choices of $\cM_0$.
By Theorem \ref{thm:BB} (4), we see that 
\[\nearbyd{\cM_U}{\alpha}\not= 0 \textup{ iff } \alpha\in \Lambda.\]
When $\alpha=0$, $\nearbyd{\cM_U}{0}$ coincides with $\Psi^{\textup{nil}}(\cM)$ in \cite[\S 2.4]{BG}.
\begin{prop}\label{prop:poneaby}
For every $\alpha\in\bbC$, we have
\begin{enumerate}
    \item $(s+\alpha)^N$ annihilates $\nearbyd{\cM_U}{\alpha}$ for some $N\gg 0$.
    \item $\nearbyd{\cM_U}{\alpha}$ is holonomic $\shD_X$-module supported on $D$; moreover, if $\cM_U$ is regular holonomic, then so is $\nearbyd{\cM_U}{\alpha}$;
    \item $\D(\nearbyd{\cM_U}{\alpha})\simeq \nearbyd{\D(\cM_U)}{-\alpha}$.
\end{enumerate}
\end{prop}
\begin{proof}
This proposition is essentially proved in \cite[\S 4.2]{BB93}. We give a proof here for completeness. 

By Theorem \ref{thm:BB} (3), we have 
\[\nearbyd{\cM_U}{-\alpha}=\dfrac{\shD_X[s]_{m_\alpha}\cM_0\cdot f^{s-k}}{\shD_X[s]_{m_\alpha}\cM_0\cdot f^{s+k}}\]
Therefore, $(s-\alpha)^N$ annihilates $\nearbyd{\cM_U}{-\alpha}$ for some $N\gg 0$ by using the $b$-function of $\cM_U$. The first statement is thus proved. 

For the second one, it is obvious that $\nearbyd{\cM_U}{\alpha}$ is supported on $D$. We then prove holonomicity. Using Theorem \ref{thm:BB} (3) one more time, we obtain a short exact sequence
\[0\to j_*(\cM_U\cdot f^\alpha)\rightarrow\dfrac{j_*(\cM_U[s]_{m_\alpha}\cdot f^s)}{(s-\alpha)^2\cdot j_*(\cM_U[s]_{m_\alpha}\cdot f^s)}\to  j_*(\cM_U\cdot f^\alpha)\to0.\]
Hence, $\dfrac{j_*(\cM_U[s]_{m_\alpha}\cdot f^s)}{(s-\alpha)^2\cdot j_*(\cM_U[s]_{m_\alpha}\cdot f^s)}$ is holonomic. By induction, we then have $\dfrac{j_*(\cM_U[s]_{m_\alpha}\cdot f^s)}{(s-\alpha)^N\cdot j_*(\cM_U[s]_{m_\alpha}\cdot f^s)}$ is holonomic. 

By Part (1), we have
\[\nearbyd{\cM_U}{-\alpha}=\nearbyd{\cM_U}{\alpha}/(s-\alpha)^N\cdot\nearbyd{\cM_U}{\alpha}.\]
Therefore, $\nearbyd{\cM_U}{-\alpha}$ is a quotient of $\dfrac{j_*(\cM_U[s]_{m_\alpha}\cdot f^s)}{(s-\alpha)^N\cdot j_*(\cM_U[s]_{m_\alpha}\cdot f^s)}$, from which we have proved the holonomicity. Regularity can be proved similarly.

The third one follows from Lemma \ref{lm:opendual} and Eq.\eqref{eq:commDL}. 
\end{proof}

We now give an alternative description of $\alpha$-nearby cycle. 
\begin{prop}\label{prop:altnb}
For each $\alpha\in\bbC$, $\nearbyd{\cM_U}{\alpha}$ is canonically isomorphic to the generalized $-\alpha$-eigenspace of $\shD_X[s]\cM_0\cdot f^{s-k}/\shD_X[s]\cM_0\cdot f^{s-k}$ with respect to the $s$-action for $k\gg 0$.
\end{prop}
\begin{proof}
Using the $b$-function, we first know that the $s$-action on $\shD_X[s]\cM_0\cdot f^{s-k}/\shD_X[s]\cM_0\cdot f^{s-k}$ admits a minimal polynomial for each $k\ge0$. Hence, we have the generalized $\alpha$-eigenspace. Since as a $\bbC[s]$-module, $\shD_X[s]\cM_0\cdot f^{s-k}/\shD_X[s]\cM_0\cdot f^{s-k}$ is supported at a finite subset of $\Spec\bbC[s]$ (determined by the $b$-function). Therefore, its $\alpha$-eigenspace is naturally 
\[(\shD_X[s]\cM_0\cdot f^{s-k}/\shD_X[s]\cM_0\cdot f^{s-k})_{m_{\alpha}}.\]
The proof is now done by Theorem \ref{thm:BB}(3).
\end{proof}

We write by $b_f(s)$ the $b$-function for $\sO_U$ along $f$. Since $\D(\sO_U)\simeq \sO_U$, we have the following well-known fact as an immediate corollary of Proposition \ref{prop:poneaby} (2):
\begin{coro}
If $\alpha$ is a root of $b_f(s)$, then $-\alpha+k$ is also a root of $b_f(s)$ for some $k\in \Z$.
\end{coro}
\begin{definition}[Beilinson]
The maximal extension of $\cM_U$ is 
\[\Xi(\cM_U)=\dfrac{j_*(\cM_U[s]_{m_0}\cdot f^s)}{s\cdot j_!(\cM_U[s]_{m_0}\cdot f^s)}.\]
\end{definition}
Using Theorem \ref{thm:BB} (3) with $\alpha=0$, we then have the following two short exact sequences
\be\label{eq:b-}
0\to j_!(\cM_U)\xrightarrow{\alpha_-} \Xi(\cM_U) \xrightarrow{\beta_-} \nearbyd{\cM_U}{0}\to 0
\ee
and 
\be\label{eq:b+}
0\to \nearbyd{\cM_U}{0}\xrightarrow{\beta_+} \Xi(\cM_U) \xrightarrow{\alpha_+} j_*(\cM_U)\to 0,
\ee
where $\beta_+$ is induced by the isomorphism
\[\nearbyd{\cM_U}{0}\simeq \dfrac{s\cdot j_*(\cM_U[s]_{m_0}\cdot f^s)}{s\cdot j_!(\cM_U[s]_{m_0}\cdot f^s)}.\] 
By construction, \eqref{eq:b-} and \eqref{eq:b+} are dual to each other.

Since $\cM|_U\simeq \cM_U$, we have natural morphisms
\[\cM\rightarrow j_*(\cM_U)\textup{ and } j_!(\cM_U)\rightarrow \cM\]
where the second one is obtained by taking duality of the first one. We then have the following commutative diagram 
\be \label{eq:diagvan}
\begin{tikzcd}
\Xi(\cM_U)\arrow[r,"\alpha_+"] & j_*(\cM_U)\\
j_!(\cM_U)\arrow[u,"\alpha_-"]\arrow[r] & \cM\arrow[u].
\end{tikzcd}
\ee
\begin{definition}[Beilinson]
The vanishing cycle of $\cM$ is
\[\vanish{\cM}\coloneqq \cH^0([j_!(\cM_U)\to\Xi(\cM_U)\oplus M\to j_*(\cM_U)])\]
where the complex is the total complex of the above double complex in degrees $-1$, 0 and 1.
\end{definition}
Using the two short exact sequences \eqref{eq:b-} and \eqref{eq:b+}, we have 
\be \label{eq:lesvc}
 \cH^i([j_!(\cM_U)\to\Xi(\cM_U)\oplus M\to j_*(\cM_U)])=0 \textup{ for } i\not=0.
\ee

We then have the morphisms of $\shD_X$-modules
\be\label{eq:canD}
c: \nearbyd{\cM}{0}\rightarrow \vanish{\cM}
\ee
given by $c(\eta)=(\beta_+(\eta),0)$, and 
\be\label{eq:varD}
v: \vanish{\cM}\rightarrow \nearbyd{\cM}{0}
\ee
given by $v(\xi,m)=\beta_-(\xi)$. Then
\[
v\circ c=s \textup{ and } c\circ v=(s, 0).\]

The above construction of $\Xi(\cM_U)$ and $\vanish{\cM}$ exactly follows the recipe in \cite{BeiGP} and $\vanish{\cM}$ coincides with $\Phi^{\textup{nil}}(\cM)$ in \cite[\S 2.4]{BG}.

The following corollary follows immediately from Proposition \ref{prop:poneaby} and \eqref{eq:lesvc}.
\begin{coro}We have:
\begin{enumerate}
    \item $\Xi(\cM_U)$ and $\vanish{\cM}$ are both holonomic; moreover, if $\cM$ is regular holonomic, then so are $\Xi(\cM_U)$ and $\vanish{\cM}$;
    \item $\D(\Xi(\cM_U))\simeq \Xi(\D(\cM_U))$;
    \item $\D(\vanish{\cM})\simeq \vanish(\D\cM)$.
\end{enumerate}
\end{coro}

\section{Twisted $\shD_X[s]$-module by Jordan block}\label{subsec:exjb}
We discuss $\shD_X[s]$-modules twisted by local systems given by Jordan blocks. We first consider a key example: Local systems of Jordan blocks on $\bbC^*$.

For $\alpha\in \bbC$ and $m\ge 1$, we define a free $\sO_\bbC[1/t]$-module
\[K^\alpha_{m}=\bigoplus_{l=0}^{m-1}\sO_\bbC[t^{-1}]e^\alpha_{l}\]
with a naturally defined connection $\nabla$ by requiring
\[\nabla e^\alpha_{l}=\frac{1}{t}(\alpha e^\alpha_{l}+e^\alpha_{l-1}),\]
where $t$ is the coordinate of the complex plane $\bbC$. The generator $e^\alpha_l$ can be understood as the formal symbol of the multi-valued function 
$t^{\alpha}\frac{\log^lt}{l!}$ and we conventionally set $e^\alpha_{-1}=0$. 

We can identify $t\nabla$ with the action of $J_{\alpha,m}$, where $J_{\alpha,m}$ is the $m\times m$ Jordan block with the eigenvalue $\alpha$. The nilpotent part of $t\nabla$ is then $J_{0,m}$, or more explicitly 
\[(t\nabla)^{\textup{nil}}
(e^\alpha_l)=e^\alpha_{l-1}.\]
It is then obvious that the multivalued $\nabla$-flat sections of $K^\alpha_{m}$ (on $\bbC^*$) are the $\C$-span of
\[\{e^{-J_{\alpha,m}\log t}\cdot e^\alpha_k\}_{k=0,\dots,m-1}.\]
We set $L^\lambda_m$ the local system of the multivalued $\nabla$-flat sections of $K^\alpha_{m}$, or equivalently
\[\DR(K^\alpha_{m})|_{\C^*}\stackrel{q.i.}{\simeq}L^\lambda_m[1],\]
where $\lambda=e^{2\pi\sqrt{-1}\alpha}$.

The monodromy action $T$ (around the origin of the complex plane counterclockwise) on $L^\lambda_m$ is given by $e^{-2\pi\sqrt{-1}J_{\alpha,m}}$. In particular 
\[\log T_u=-2\pi\sqrt{-1}J_{0,m}\]
where $T_u$ is the uniportent part of $T$ in the Jordan-Chevalley decomposition.

By construction, we have a direct system of $\shD$-modules
\[\cdots\to K^\alpha_m\rightarrow K^\alpha_{m+1} \to\cdots .\]
Applying $\DR$, we then obtain a direct system of local systems
\[\cdots\to L^\lambda_m\rightarrow L^\lambda_{m+1} \to\cdots. \]

We now define the $\shD_X[s]$-module  
\[\cN^{\alpha,k}_{m}\coloneqq\bigoplus_{l=0}^{m-1}\shD_X[s]\cM_0\cdot f^{s-k}\otimes e^{\alpha}_l,\]
by assigning the $s$-action by 
\[s\cdot (\eta e^{\alpha}_l)=(s+\alpha)\eta e^{\alpha}_l-\eta e^{\alpha}_{l-1}\]
for $\eta$ a section of $\shD_X[s]\cM_0\cdot f^{s-k}$. We therefore have a direct symstem of $\shD_X[s]$-modules 
\[\cdots\to\cN^{\alpha,k}_{m}\rightarrow \cN^{\alpha,k}_{m+1}\to \cdots.\]

Let $\iota: X\hookrightarrow Y=X\times \bbC$ be the graph embedding of $f$, i.e. 
\[\iota(x)=(x,f(x)).\]
By identifying $s$ with $-\partial_tt$, we have a natural injection 
\be\label{eq:grampheb}
\cN^{\alpha,k}_{m}\hookrightarrow \iota_{+}(j_*(\cM_U))\otimes_{\sO_Y}p_1^* K^{-\alpha}_m,
\ee
where $\iota_{+}$ denotes the $\shD$-module direct image functor and $p_1: Y\to \bbC$ the projection (cf. \cite[\S2.4]{BMS}). The above injection is indeed a morphism of log $\shD$-modules with the log structure along $X\times\{0\}$; see \cite[\S2]{WZ} for definitions. 
\begin{lemma}\label{lm:btwist}
Assume that $b(s)$ is the $b$-function of $\cM_U$ along $f$. Then 
\[{b(s+\alpha)}^m \textup{ annihilates } \cN^{\alpha,0}_{m}/\cN^{\alpha,-1}_{m}.\]
\end{lemma}
\begin{proof}
By construction, we have a short exact sequence of $\shD_X[s]$-modules
\[0\to\cN^{\alpha,k}_{m-1}\rightarrow \cN^{\alpha,k}_{m}\rightarrow \cQ\to 0.\]
We also know that 
\[\cN^{\alpha,k}_1\simeq \shD_X[s]\cM_0\cdot f^{s-k+\alpha}\simeq \cQ.\]
By substituting $s+\alpha$ for $s$, we know $b(s-k+\alpha)$ annihilates $$\shD_X[s]\cM_0\cdot f^{s-k+\alpha}/\shD_X[s]\cM_0\cdot f^{s-k+1+\alpha}.$$ Therefore, we obtain the required statement by induction. 

\end{proof}
\begin{prop}\label{prop:comptwist}
For each $\alpha\in\bbC$, there exists $k_0>0$ so that for all $k\ge k_0$
\[\DR(\cN^{\alpha,k}_m\xrightarrow{s}\cN^{\alpha,k}_m)\stackrel{q.i.}{\simeq}\iota_*Rj_*(\DR(\cM_U)\otimes f^{-1}_0L^{1/\lambda}_m)\]
and
\[\DR(\cN^{\alpha,-k}_m\xrightarrow{s}\cN^{\alpha,-k}_m)\stackrel{q.i.}{\simeq}\iota_!j_!(\DR(\cM_U)\otimes f^{-1}_0L^{1/\lambda}_m)\]
for all $m$, where the complex $[\cN^{\alpha,k}_m\xrightarrow{s}\cN^{\alpha,k}_m]$ is in degrees $-1$ and $0$, and $\lambda=e^{2\pi\sqrt{-1}\alpha}$.
\end{prop}
\begin{proof}
Under the inclusion \eqref{eq:grampheb}, we have 
\[\cN^{\alpha,k}_m|_U\simeq\iota_{+}j_*(\cM_U)\otimes_{\sO_Y}p_1^* K^{-\alpha}_m|_U\simeq \iota_{+}j_*(\cM_U)|_U\otimes_\bbC p_1^{-1}L^{1/\lambda}_m.\]
By using Theorem \ref{thm:BB}(3), we have that 
\[\cN^{\alpha,k}_m\otimes^L_{\bbC[s]}\bbC_0 \simeq  (\cN^{\alpha,k}_m)_{m_0}\otimes^L_{\bbC[s]_{m_\alpha}}\bbC_\alpha\simeq j_*(\cN^{\alpha,k}_m|_U)\otimes^L_{\bbC[s]}\bbC_0,\]
where $(\cN^{\alpha,k}_m)_{m_0}$ is the localization of $\cN^{\alpha,k}_m$ at $m_0$, the maximal ideal of $0\in \bbC$. Moreover, since we identify $s$ with $-\partial_tt$, we further have 
\[j_*(\cN^{\alpha,k}_m|_U)\otimes^L_{\bbC[s]}\bbC_0\stackrel{q.i.}{\simeq} [\iota_{+}j_*(\cM_U)\otimes_{\sO_Y}p_1^* K^{-\alpha}_m\xrightarrow{\partial_t} \iota_{+}j_*(\cM_U)\otimes_{\sO_Y}p_1^* K^{-\alpha}_m].\]
By using the Koszul decompostions of $de$ $Rham$ complexes (see for instance \cite[\S4.1]{Wnv}), we therefore have
\[\DR(\cN^{\alpha,k}_m\xrightarrow{s}\cN^{\alpha,k}_m)\stackrel{q.i.}{\simeq}\DR_Y(\iota_{+}j_*(\cM_U)\otimes_{\sO_Y}p_1^* K^{-\alpha}_m),\]
where the second $\DR$ is taken over the ambient space $Y$. Since $$\iota_{+}(j_*(\cM_U))\otimes_{\sO_Y}p_1^* K^{-\alpha}_m$$ is regular holonomic, we also naturally have 
\[\DR_Y(\iota_{+}j_*(\cM_U)\otimes_{\sO_Y}p_1^* K^{-\alpha}_m)\stackrel{q.i.}{\simeq} Rj_{Y*}(\iota^o
_*\DR(\cM_U)\otimes_\bbC p_1^{-1}L^{1/\lambda}_m),\]
where $j_Y$ and $\iota^o$ are as in the following diagram 
\[
\begin{tikzcd}
U \arrow[r,hook,"\iota^o"]\arrow[d,hook,"j"] & U_Y=Y\setminus X\times\{0\}\arrow[d, hook,"j_Y"] \\
X \arrow[r,hook,"\iota"] & Y;
\end{tikzcd}
\]
see \cite[Chapter V.4]{Bj}. By projection formula for local systems (\cite[Proposition 2.5.11]{KS}), we further have 
\[Rj_{Y*}(\iota^o
_*\DR(\cM_U)\otimes_\bbC p_1^{-1}L^{1/\lambda}_m)\stackrel{q.i.}{\simeq} \iota_*(Rj_*(\DR(\cM_U)\otimes_\bbC f^{-1}_0L^{1/\lambda}_m))\]
and the first quasi-isomorphism is obtained. 

The second quasi-isomorphism can be obtained similarly. The choice of $k_0$ only depends on $\alpha$ and the roots of the $b$-function annihilating $\cN^{\alpha,0}_{m}/\cN^{\alpha,-1}_{m}$. We therefore can choose a uniform $k_0$ working for all $m$ by Lemma \ref{lm:btwist}.
\end{proof}

\section{Nearby cycles for perverse sheaves via Jordan blocks}\label{sec:nvp}
In this section, we define nearby and vanishing cycles via local systems given by Jordan blocks on $\C^*$, the punctured complex plane. We keep the notations introduced in \S \ref{subsec:bf}.
Assume that $K$ is a $\bbC$-perverse sheaf on $X$. 
\begin{definition}
For $\lambda\in \bbC^*$, the $\lambda$-nearby cycle of $K$ is 
\[\psi_{f,\lambda}(K)\coloneqq\varinjlim_m i^{-1}Rj_*(j^{-1}K\otimes f^{-1}_0L^{1/\lambda}_m).\]
The vanishing cycle is 
\[\vanishp{K}\coloneqq\Cone(i^{-1}K\rightarrow \psi_{f,1}(K))\]
where the morphism $i^{-1}K\rightarrow \Psi_{f,1}(K)$ is induced by the natural map 
$$K\to Rj_*(j^{-1}K).$$
The monodromy action $T$ of $L^{1/\lambda}_m$ induces the monodromy action on $\psi_{f,\lambda}(K)$ for each $\lambda$, denoted also by $T$. 
We then have the induced monodromy action $T$ on $\vanishp{K}$, by requiring $T$ acting on $i^{-1}K$ identically. 
\end{definition}
By construction, we have the tautological triangle 
\[i^{-1}K\rightarrow \psi_{f,1}(K)\xrightarrow{\can}\vanishp{K}\xrightarrow{+1} \]
with the induced canonical map
\[\can:\psi_{f,1}(K)\rightarrow\vanishp{K}.\]
We define
\[\Var: \vanishp{K}\rightarrow\psi_{f,1}(K)\]
by 
\[\Var\coloneqq(0, -\dfrac{\log T_u}{2\pi\sqrt{-1}})=(0, J_{0,\infty})\]
where $\displaystyle J_{0,\infty}=\varinjlim_m J_{0,m}$. 
\begin{remark}
Using these definitions, one can prove the perversity of $\nearbydp{K}{\lambda}$ and $\vanishp{K}$ (with a shift of cohomological degrees) directly. Let us refer to \cite{Rei} for the proof of this point and other related results.
\end{remark}

\section{The proof of Theorem \ref{thm:maincom}}
In this section, we prove Theorem \ref{thm:maincom}. Before we start, the following preliminary result about infinite Jordan blocks is needed.
\begin{lemma}\cite[6.4.5 Lemma]{Bj}
\label{lm:inftyjb}
Let $W$ be a $\C$-vector space, and let $\varphi$ be a $\C$-linear operator on $W$ admitting a minimal polynomial. Set $W_\infty=\bigoplus_{k=0}^\infty W\otimes e_k$ and define
\[\varphi_\infty(w\otimes e_k)=(\varphi-\alpha)w\otimes e_k-w\otimes e_{k-1}
\] 
for $w\in W$ $($assume $e_{-1}=0$$)$. Then $\varphi_\infty$ is surjective and $\ker(\varphi_\infty)\simeq W_\alpha$, where $W_\alpha$ is the generalized $\alpha$-eigenspace. 
\end{lemma}
\begin{proof}
Define a map $W_\alpha\rightarrow\ker(\varphi_\infty)$ by
\[w\mapsto \sum_{i\ge 0}(\varphi-\alpha)^iw\otimes e_i.\]
Clearly, this map is an isomorphism. 

We then prove surjectivity. If $w\in W_\alpha$,
then 
\[\varphi_\infty(-\sum_{i> j}(\varphi-\alpha)^{i-j-1}w\otimes e_i)=w\otimes e_j.\]
If $w\in W^{\bot}_\alpha$, then 
then 
\[\varphi_\infty(\sum_{i=1}^j(\varphi-\alpha)^{-i}w\otimes e_{j-i+1})=w\otimes e_j.\]
Therefore, the surjectivity follows.
\end{proof}

Using the above lemma and Proposition \ref{prop:altnb}, we immediately have:
\begin{coro}\label{cor:limnba}
For a holonomic $\shD_U$-module $\cM_U$ and some $\alpha\in \bbC$, there exists $k\gg 0$ so that 
\[\nearbyd{\cM_U}{\alpha}[1]\stackrel{q.i.}{\simeq}\varinjlim_m(\dfrac{\cN^{\alpha,k}_{m}}{\cN^{\alpha,-k}_{m}}\xrightarrow{s} \dfrac{\cN^{\alpha,k}_{m}}{\cN^{\alpha,-k}_{m}})\]
\end{coro}


\begin{proof}[Proof of Theorem \ref{thm:maincom}]
We assume $\cM$ a regular holonomic $\shD_X$-module and write $\cM_U=\cM|_U$. We first prove the comparison of nearby cycles. 

Since $\DR$ and the direct limit functor commute (since $\DR$ is identified with $\omega_X\otimes^L_\shD\bullet$, one can apply \cite[Corollary 2.6.17]{Weib}), we have
\[\DR(\nearbyd{\cM_U}{\alpha}[1])\simeq i_*\nearbydp{\DR(\cM_U)}{\lambda}\]
by Corollary \ref{cor:limnba} and Proposition \ref{prop:comptwist}. 

One can check that the operator $J_{0,
\infty}$ on $\varinjlim_m \nearbyd{\cN^\alpha_m}{0}$ corresponds to the action $(s+\alpha)$ on $\nearbyd{\cM_U}{\alpha}$ under the quasi-isomorphism in Corollary \ref{cor:limnba} by Lemma \ref{lm:inftyjb}. Taking $\DR$, the operator $J_{0,\infty}$ becomes $-\dfrac{\log T_u}{2\pi\sqrt{-1}}$ on $\nearbydp{\DR(\cM_U)}{\lambda}$ by the construction of the monodromy operator $T$. Or equivalently, the monodromy operator $T$ corresponds to 
\[T\simeq \DR(e^{-2\pi \sqrt{-1}s})=\DR(\lambda\cdot e^{-2\pi \sqrt{-1}(s+\alpha)}).\]
We thus have proved the comparison for nearby cycles in Theorem \ref{thm:maincom}. 

We now prove the comparison for vanishing cycles. For simplicity, we write by $\cA^\bullet$ the complex
\[[\Xi(\cM_U)\to j_*(\cM_U)]\]
in degrees 0 and 1,
by $\cB^\bullet$ the complex 
\[[j_!(\cM_U)\to\Xi(\cM_U)\oplus M\to j_*(\cM_U)]\]
by $\cC^\bullet$ the complex 
\[[j_!(\cM_U)\to \cM]\]
in degrees -1 and 0.
Then we have a triangle
\[\cC^\bullet\rightarrow \cA^\bullet[1]\rightarrow \cB^\bullet[1] \xrightarrow{+1}.\]
Considering the two short exact sequences \eqref{eq:b-} and \eqref{eq:b+}, we see that 
$\vanish{\cM}\stackrel{q.i.}{\simeq}\cB^\bullet$
and
\[[\nearbyd{\cM}{0}\xrightarrow{c}\vanish{\cM}]\stackrel{q.i.}{\simeq}[\cA^\bullet\to \cB^\bullet].\]
We hence have 
\be\label{eq:vancqi}
\vanish{\cM}[1]\stackrel{q.i.}{\simeq} \Cone(\cC^\bullet\rightarrow\nearbyd{\cM_U}{0}[1]).
\ee
Since 
\[\DR(\cC^\bullet)\stackrel{q.i.}{\simeq} i_*i^{-1}\DR(\cM),\]
we  have 
\[\DR(\nearbyd{\cM_U}{0}[1]\xrightarrow{c} \vanish\cM[1])\simeq i_*\big(\nearbydp{\DR(\cM)}{0}\xrightarrow{\can} \vanishp{\DR(\cM)}\big).\]
Using the short exact sequence \eqref{eq:b+}, we see that the $\nearbyd{\cM_U}{0}$ in $$\Cone(\cC^\bullet\rightarrow\nearbyd{\cM_U}{0}[1])$$ 
has a $s$-twist. Therefore, under the quasi-isomorphism \eqref{eq:vancqi}
\[v:\Cone(\cC^\bullet\rightarrow\nearbyd{\cM_U}{0}[1])\to  \nearbyd{\cM_U}{0}[1]\]
is given by $(0,s)$.
But the $s$ action on $\nearbyd{\cM_U}{0}$ is $J_{0,\infty}$ after taking $\DR$. Therefore, we obtain
\[\DR(\vanish{\cM}[1]\xrightarrow{v}\nearbyd{\cM}{0}[1])\simeq i_*(\vanishp{\DR(\cM)}\xrightarrow{\Var}\nearbydp{\DR(\cM)}{0}).\]
\end{proof}
\bibliographystyle{amsalpha}
\bibliography{mybib}
\end{document}